\documentclass[10pt, reqno]{amsart}
\usepackage{amsmath,amssymb,amsfonts,amscd,hyperref,color}
\usepackage[latin1]{inputenc}

\usepackage[abs]{overpic}
\usepackage{verbatim}

\newcommand{\Hmm}[1]{\leavevmode{\marginpar{\tiny%
$\hbox to 0mm{\hspace*{-0.5mm}$\leftarrow$\hss}%
\vcenter{\vrule depth 0.1mm height 0.1mm width \the\marginparwidth}%
\hbox to
0mm{\hss$\rightarrow$\hspace*{-0.5mm}}$\\\relax\raggedright #1}}}

\newtheorem{theorem}{Theorem}
\newtheorem{lemma}[theorem]{Lemma}
\newtheorem{proposition}[theorem]{Proposition}

\theoremstyle{definition}
\newtheorem{example}[theorem]{Example}
\newtheorem{remark}[theorem]{Remark}

\numberwithin{equation}{section}
\renewcommand{\epsilon}{\varepsilon}
\renewcommand{\phi}{\varphi}
\newcommand{\NN}{\mathbb{N}}
\newcommand{\RR}{\mathbb{R}}
\newcommand{\CC}{\mathbb{C}}

\newcommand{\abs}[1]{\left\lvert #1\right\rvert} 

\begin{document}
	\title{An improved discrete $ p $-Hardy inequality}
	
	\author{Florian Fischer}
	\address{Florian Fischer, Institute of Mathematics, University of Potsdam, Germany}
	\email{florifis@uni-potsdam.de}
	\author{Matthias Keller}
	\address{Matthias Keller, Institute of Mathematics, University of Potsdam, Germany}
	\email{matthias.keller@uni-potsdam.de}
	\author{Felix Pogorzelski}
	\address{Felix Pogorzelski, Institute of Mathematics, University of Leipzig, Germany}
	\email{felix.pogorzelski@math.uni-leipzig.de}

\begin{abstract}
	We improve the classical discrete Hardy inequality for $ 1<p<\infty $ for functions on the natural numbers. For integer values of $ p $ the Hardy weight is shown to have a series expansion with strictly positive coefficients.
	\\[2mm]
	\noindent  2000  \! {\em Mathematics  Subject  Classification.}
	Primary  \! 26D15; Secondary  47J05 
\end{abstract}

\maketitle
\section{Introduction and Main Result}
In 1918 Hardy was looking for a simple and elegant proof of Hilbert's theorem in the context of the convergence of double sums, \cite{Hardy20}. Although it is not explicitly mentioned, the paper contains  the essential argument for his then famous inequality. In a letter to Hardy in 1921,  \cite{Landau21}, Landau gave a proof  with the sharp constant  
\[\sum_{n=1}^\infty a_n \geq \biggl(\frac{p-1}{p} \biggr)^p \sum_{n=1}^\infty \biggl( \frac{a_1+a_2+\ldots+ a_n}{n} \biggr)^p\]
for $p>1$ where $(a_n)$ is an arbitrary sequence of non-negative real numbers. This inequality was first highlighted in \cite{HLP34} and  is referred to as  a \emph{$p$-Hardy inequality}. Since then various proofs of this inequality were given, where short and elegant ones are due to Elliott \cite{Elliott26} and Ingham, see \cite[p. 243]{HLP34} and most recently by Lef\`evre \cite{Lef19}. See also \cite{KMP06} for a beautiful historical survey about the origins of Hardy's inequality.

It is not hard to see that the inequality above can be derived from the following inequality for compactly supported $\phi\in C_c(\NN)$ with $\phi(0)=0$
\begin{align*}
\sum_{n=1}^\infty \abs{\phi(n)-\phi(n-1)}^p \geq  \sum_{n=1}^\infty w_{p}^{H}(n){\abs{\phi(n)}^p},
\end{align*}
where  
\begin{align*}
w_{p}^{H}(n)=\biggl(\frac{p-1}{p} \biggr)^p\frac{1}{n^p}.
\end{align*}
 In this paper, we  prove the following improvement of this inequality.

\begin{theorem}\label{thm}
	Let $p> 1$. Then, for all $\phi\in C_c(\NN)$ with $\phi(0)=0$,
	\begin{align*}
	\sum_{n=1}^\infty \abs{\phi(n)-\phi(n-1)}^p \geq \sum_{n=1}^\infty w_{p}(n)\abs{\phi(n)}^p,
	\end{align*}
	where $w_{p}$ is a strictly positive function given by
	\begin{align*}
	w_p(n)=\left( 1-\left(1-\frac{1}{n}\right)^{\frac{p-1}{p}}\right)^{p-1}-\left( \left(1+\frac{1}{n}\right)^{\frac{p-1}{p}}-1\right)^{p-1}.
	\end{align*}
	Furthermore, we have  for all  $ n\in\NN  $
	\[ w_p(n)>w_{p}^{H}(n).\] 
	Moreover, for integer $ p\ge 2 $, we have $ w_{p}(n)=\sum_{k\in 2\NN_{0}}c_{k} n^{-k-p} $ with $ c_{k}>0 $.
\end{theorem}

\begin{example} The case $p = 2$ was already covered in \cite{KePiPo3}. In this case one gets	 
	 $ 	 w_2(1) =  2 - \sqrt{2} $ and for $ n\ge 2 $
	 \begin{align*}
	 w_2(n) &= \,\, - \sum_{k \in 2\NN} \binom{1/2}{k}  \frac{2}{n^k} 
	 \, =\,\, \frac{1}{4} \frac{1}{n^2} + \frac{5}{64} \frac{1}{n^4} + \frac{21}{512} \frac{1}{n^6} + \frac{429}{16384} \frac{1}{n^8}+ \ldots
	 \end{align*}
	 In the case
	 $p=3$, one obtains
$ 	 w_3(1) = \,\, 1 - (2^{2/3}-1)^2 $
	 and for $ n\ge 2 $
	 \begin{align*}
	 w_3(n) 	 &=  \sum_{k \in 2\NN + 1} \left( 2\,\binom{2/3}{k} -  \binom{4/3}{k}  \right) \,\frac{2}{n^k} = \frac{8}{27} \frac{1}{n^3} + \frac{8}{81} \frac{1}{n^5} + \frac{112}{2187} \frac{1}{n^7} + \ldots 
	 \end{align*}
	 In the case $p=4$, one gets
$ 	 w_4(1) = \,\, 1 - (2^{3/4}-1)^3  $ and for $ n\ge 2 $	 
	 \begin{align*}
	 {w_4(n)} 	 &=   \sum_{k \in 2\NN + 2} \left(
	 3 \binom{3/2}{k} -3 \binom{3/4}{k} 
	 - \binom{9/4}{k}  \right) \,\frac{2}{n^k} \\
	 &= \quad \frac{81}{256} \frac{1}{n^4} + \frac{891}{8192} \frac{1}{n^6} + \frac{58653}{1048576} \frac{1}{n^8} +  \ldots 
	 \end{align*}
	 For general $ p>1 $ one obtains the  asymptotics
	 \begin{align*}
	 w_{p}(n)=\left(\frac{p - 1}{pn} \right)^p \bigl(1+a_{p}(n)\bigr),
	 \end{align*}
	 where
	 \begin{align*}
a_{p}(n)	=  \left(\frac{ 3}{ 8} - \frac{1}{8 p}\right) \frac{1}{n^2} + \left(\frac{215 p^3 - 38 p^2 - 31 p + 6}{1152 p^3} \right)\frac{1}{n^4} +  O\left(\frac{1}{n^6}\right).
	 \end{align*}
	 From this formula it is clear that $ w_{p}(n) $ is strictly larger than the classical Hardy weight for large $ n $. Note however that the theorem above states that $ a_{p}(n)>0 $ at all places $n \in \NN$. It is not hard to check that $ a_{p}(n) $ can be expanded into a power series with respect to $ 1/n $
	 where all odd  coefficients  vanish. Theorem~\ref{thm} states that for integer $ p\ge 2 $ these coefficients are positive. We conjecture that all these coefficients are strictly positive for all $ p>1 $. 
\end{example}


\section{Proof of the Hardy inequality}\label{sec:Hardy}



The \emph{combinatorial $p$-Laplacian} $\Delta_p$ for real valued functions on $ \NN_{0} $ is given by
\[\Delta_p f(n)=\sum_{m=n\pm1}\mathrm{sgn}\left( f(n)-f(m) \right)\abs{f(n)-f(m)}^{p-1}\]
for all functions $f$ and $ n\ge 1 $, where $ \mathrm{sgn} $ is the function which takes the value $ -1 $ on $ (-\infty,0) $, the value $ 1 $ on $ (0,\infty) $ and $ 0 $ at $ 0 .$ 


The following proposition is needed in order to show that the weight $w_p$ is in fact a $p$-Hardy weight.
\begin{proposition}\label{prop}
	Let $p> 1$ and let $u\colon\NN_0\to[0,\infty)$ be strictly positive on $\NN$ and such that $u(0) = 0$. Suppose that $w\colon \NN\to \RR$ satisfies $\Delta_p u=wu^{p-1}$ on $\NN$. Then for all $\phi\in C_c(\NN)$ with $\phi(0)=0$ we have
	\[	\sum_{n\in\NN}	\abs{\phi(n)-\phi(n-1)}^p\geq \sum_{n\in\NN}w(n)\abs{\phi (n)}^p.\]
\end{proposition} 
The  proof is follows along the lines of the proof of Proposition~2.2 in \cite{FS08}.
\begin{proof}
	Let $p> 1$. From Lemma 2.6 in \cite{FS08}, we obtain for all $0\leq t\leq 1$ and $a\in \CC$
	\begin{align*}
	\abs{a-t}^p \geq (1-t)^{p-1}(\abs{a}^p-t).
	\end{align*}	
	Let $ w $ be such that $\Delta_p u=wu^{p-1}  $ and $ \psi\in C_{c}(\NN) $. We assume for a moment that $m,n \in \NN$ are such that $u(n) \geq u(m)$ and $\psi(m) \neq 0$.
	We apply the above inequality with the choice $t=u(m)/u(n)$ and $a=\psi(n)/\psi(m)$ in order to obtain
	\[
	\big| (u\psi)(n) - (u\psi)(m) \big|^p \geq \big|u(n) - u(m) \big|^{p-1} \big( |\psi(n)|^p u(n) - |\psi(m)|^p u(m) \big).
	\]
	Further, since $u^{p}(n) \geq |u(n) - u(m)|^{p-1} u(n)$, the above inequality remains true even if $\psi(m) = 0$. Summing over $\NN$, we obtain
	\begin{align*}
\lefteqn{	\sum_{n\in\NN}	\abs{(u\psi)(n)-(u\psi)(n-1)}^p}\\
	&\ge\sum_{n\in\NN}\mathrm{sgn}(u(n)-u(n-1))\abs{u(n)-u(n-1)}^{p-1}\bigl(\abs{\psi(n)}^pu(n)-\abs{\psi(n-1)}^pu(n-1) \bigr) \\
	&=\sum_{n\in\NN}\abs{\psi(n)}^pu(n)\Delta_p u(n).
	\end{align*}
	Note that the latter equality follows from rearranging the involved sums while recalling that $u(0)= 0$. 
	Using the assumption $\Delta_p u=wu^{p-1}  $ we arrive at
	\begin{align*}
	\ldots&=\sum_{n\in\NN}w(n)\abs{(\psi u)(n)}^p.
	\end{align*}
	With $ \phi=u\psi $ and by  strict positivity of $ u $ on $\NN$, we infer the statement.
\end{proof}

Next we show that for the  weight $ w_{p} $ on $ \NN $ taken from Theorem~\ref{thm} 
\begin{align*}
	w_p(n)=\left( 1-\left(1-{1}/{n}\right)^{{(p-1)}/{p}}\right)^{p-1}-\left( \left(1+{1}/{n}\right)^{{(p-1)}/{p}}-1\right)^{p-1},
\end{align*}
there is a suitable positive function $ u $ such that $ \Delta_{p}u =w_{p}u^{p-1}$.

\begin{proposition}\label{p:w_p}	Let $p> 1$. Then, the function $ u\colon \NN_0 \to [0,\infty)$, $ u(n)=n^{(p-1)/p} $ satisfies
	\begin{align*}
	\Delta_{p}u=w_{p}u^{p-1}\qquad \mbox{on }\NN.
	\end{align*}
\end{proposition}
\begin{proof}
	One directly checks that for all $n\in \NN$
	\begin{align*}
\frac{\Delta_{p}u(n)}{u^{p-1}(n)}	=\frac{\Delta_p n^{(p-1)/p}}{n^{(p-1)^{2}/p}}=w_{p}(n)
	\end{align*}
	which immediately yields the statement.	
\end{proof}

Combining the two propositions above already yields the $ p $-Hardy inequality with the weight $ w_{p} $. Next we show that $ w_{p} $ is strictly larger than the classical Hardy weight $ w_{p}^{H}(n)=\bigl({(p-1)}/{p }\bigr)^{p}{n^{-p}}$ for all  $ n\in\NN $.

\section{Proof of  $ w_p>w_{p}^{H}$}\label{sec:w_p>w_H}
In this section we show that the weight \[ w_p(n)=\left( 1-\left(1-\frac{1}{n}\right)^{\frac{p-1}{p}}\right)^{p-1}-\left( \left(1+\frac{1}{n}\right)^{\frac{p-1}{p}}-1\right)^{p-1} \] 
from the main theorem, Theorem~\ref{thm}, is strictly larger than the classical $p$-Hardy weight \[ w^H_p(n)=\left(\frac{p-1}{p} \right)^p\frac{1}{n^p}. \]
In fact, for fixed $ p\in(1,\infty) $, we  analyze the function $ w\colon[0,1]\to [0,\infty) $
\begin{align*}
w(x)&=\left(1-(1-x)^{1/q}\right)^{p-1}-\left((1+x)^{1/q}-1\right)^{p-1}
\end{align*}
for $x \in [0,1/2]$ and $x=1$, where $q\in (1,\infty)$ is  such that $1/p + 1/q=1$. Specifically, we  show 
\[w(x)> \left(\frac{x}{q}\right)^{p}.\]
The case $ x=1 $ is simple and is treated at the end of the section. The proof for $ x\leq 1/2 $ is also elementary but more involved. We proceed by bringing $ w_{p} $ into form for which we then analyze its parts. This will be eventually done by a case distinction depending on $ p $.

Recall the binomial theorem for $r \in[0,\infty)$ and $0\leq x \leq 1$
\[ (1\pm x)^{r}= \sum_{k=0}^\infty \binom{r}{k}(\pm 1)^kx^k\]
where $\binom{r}{0}=1, \binom{r}{1}=r$ and $\binom{r}{k}=r(r-1)\cdots (r-k+1)/k!$ for $k\geq 2$ which is derived from the Taylor expansion of the function $ x\mapsto (1\pm x)^{r} $. Applying this formula to the function $ w $ from above we obtain
\begin{align*}
w(x)&=\left(-\sum_{k=1}^{\infty}\binom{{1/q }}{ k} (-x)^{k}\right)^{p-1}-\left(\sum_{k=1}^{\infty}\binom{{1/q }}{ k}x^{k}\right)^{p-1}\\
&=\left(\frac{x}{q}\right)^{p-1}\left(\left(q\sum_{k=0}^{\infty}\binom{{1/q }}{ k+1} (-x)^{k}\right)^{p-1}-\left(q\sum_{k=0}^{\infty}\binom{{1/q }}{ k+1}x^{k}\right)^{p-1}\right)
\end{align*}
To streamline notation we set
\begin{align*}
g(x)=q\sum_{k=1}^{\infty}\binom{{1/q }}{ k+1}x^{k}.
\end{align*}
Note that since $q\abs{\binom{1/q}{1}}=1$ and $q\abs{\binom{1/q}{k}}<1$ for $k\geq 2$, we have ${0<  |g(\pm x)|<1}$ for $0\le x\le 1/2$. Thus, we can apply the binomial theorem to $\bigl(1+g(\pm x)\bigr)^{p-1}$ in order to get
\begin{align*}
w(x)
&=\left(\frac{x}{q}\right)^{p-1}\Bigl(\bigl(1+g(-x)\bigr)^{p-1} - \bigl(1+g(x)\bigr)^{p-1}\Bigr)\\
&=\left(\frac{x}{q}\right)^{p-1}
\left(\sum_{n=0}^\infty\binom{p-1}{n}\bigl(g^{n}(-x)-g^{n}(x)\bigr)\right)
\\ &=\left(\frac{x}{q}\right)^{p-1}
\left(\binom{p-1}{1}\bigl(g(-x)-g(x)\bigr)+\sum_{n=2}^{\infty}\binom{p-1}{n}\bigl(g^{n}(-x)-g^{n}(x)\bigr)\right)
\end{align*}
Thus, we have to show that the second factor on the left hand side is strictly larger than $ x/q $.
Using $ q=p/(p-1) $  we compute the first term in the parenthesis on the left hand side
\begin{align*}
\binom{p-1}{1}&\bigl(g(-x)-g(x)\bigr)=q(p-1)\sum_{k=1}^{\infty}\binom{{1/q }}{ k+1}\left((-x)^{k} - x^{k}\right)\\
&=\frac{q(p-1)(1/q)(1/q-1)}{2}(-2x)+q(p-1)\sum_{k=2}^{\infty}\binom{{1/q }}{ k+1}\left((-x)^{k} - x^{k}\right)
\\
&=\frac{x}{q}-2p\sum_{k\in 2\mathbb{N}+1}\binom{{1/q }}{ k+1}x^{k}\\
&=\frac{x}{q}+E_{p}(x)
\end{align*}
and we note that  since $ -2p \binom{{1/q }}{ k+1}> 0$ for odd $ k $ \[  E_{p}(x)>0  \]
for $ x>0 $.
So, it remains to show that for  the term
\begin{align*}
F_{p}(x)=\sum_{n=2}^{\infty}\binom{p-1}{n}\bigl(g^{n}(-x)-g^{n}(x)\bigr)
\end{align*}
we have for $0<x\leq 1/2  $
\begin{align*}
E_{p}(x)+F_{p}(x)> 0.
\end{align*}
Specifically, we then get with the substitution $ x=1/n $
\begin{align*}
w_{p}(n)=w(1/n)=\left(\frac{1}{nq}\right)^{p-1}
\left(\frac{1}{nq}+E_{p}(1/n)+F_{p}(1/n)\right)>\frac{1}{(nq)^p}=w_{p}^{H}(n)
\end{align*}
for $ n\ge 2 $.

\begin{remark}
	It is not hard to see that $ F_{p}\ge 0 $ whenever $ p\in\NN $ is integer valued. Indeed, $ g(-x)\ge g(x) $ as all terms in the sum $ g(-x) $ are positive since   $ - \binom{{1/q }}{ k+1}\ge 0$ for odd $ k $, while the terms in $ g(x) $ alternate, (they are positive for even $ k $ and negative for odd $ k $). Moreover for positive integers $ p $ the binomial coefficients $ \binom{p-1}{n} $ are positive.  Thus, the Hardy weight we computed is larger than the classical one for integer $ p $.
\end{remark}

Let us now turn to the proof of \[E_{p}(x)+F_{p}(x)>0 \] for $ p \in(1,\infty)$ and $0< x\leq1/2$.

We collect the following basic properties of the function $ g $ which were partially already discussed above and will be used subsequently.
\begin{lemma}\label{l:prop_g}For $ p\in(1,\infty) $ and $ 0<x\le  1/2 $, we have
	\begin{align*}
	-1<g(x)< 0< -g(x)< g(-x)<1.
	\end{align*}	
\end{lemma}
\begin{proof}
	The function $ g  $ is given by $ g(x)=q\sum_{k=1}^{\infty}\binom{{1/q }}{ k+1}x^{k} $. Since $ q>1 $, the coefficients $ b_k=q\binom{{1/q }}{ k+1} $ are negative for odd $ k $ and positive for even $ k $. Furthermore,  the sequence $ (|b_{k} |)$ takes values strictly less than $ 1 $ and decays  monotonically. Thus, the asserted inequalities follow easily.
\end{proof}

We distinguish the following three cases depending on $ p $ for which the arguments are quite different:
\begin{itemize}
	\item $ p $ lies between an odd an an even number with the subcases:
	\begin{itemize}
		\item[$ \bullet $] $ p\in[3,\infty) $
		\item[$ \bullet $] $ p\in (1,2] $
	\end{itemize}
	\item $ p $ lies between an even an odd number.
\end{itemize}

We start with investigating the case of $ p $ lying between an odd and an even number. 
To this end we consider  two subsequent summands as they appear in the sum given by $ F_{p} $ and show that they are positive. (Indeed the sum in $ F_{p} $ starts at $ n=2 $ but we also consider the corresponding term for $ n=1 $.)


\begin{lemma}\label{l:diff} Let $ p  $ be such that there is 
	$k\in\NN$  with $ 2k-1\leq p\leq 2k $. Then, for all $0< x\leq 1/2$ and odd $ n\in 2\NN-1 $
	\begin{align*}
	\binom{{p-1 }}{ n}&\bigl(g^{n}(-x)-g^{n}(x)\bigr)+\binom{{p-1 }}{ n+1}\bigl(g^{n+1}(-x)-g^{n+1}(x)\bigr)\ge 0.
	\end{align*}
\end{lemma}
\begin{proof}
	Let $n\in \NN$. Note that
	\[\binom{{p-1 }}{ n} = \frac{(p-1)(p-2)\cdots (p-n)}{n!}  \ge0 \]
	for $ n \le 2k-1 \le p$ and for $ n\ge p $ with $ n\in 2\NN -1 $. Moreover, for $ n\in 2\NN-1 $ with $ n\ge  p $, we have $ \binom{{p-1 }}{ n+1} \leq 0$. Specifically, $ \binom{{p-1 }}{ n}  $ has alternating signs for $ n\ge p $. 
	
	From now on let $ n \in 2\NN-1$ with $ n\ge p $. Then, 
	\begin{align*}
	\binom{{p-1 }}{ n}\ge -\binom{{p-1 }}{ n+1} \ge0.
	\end{align*}
	From Lemma~\ref{l:prop_g} we know $ -g(x)>0 $ and hence, for odd $ n\in 2\NN-1 $, we have
	\begin{align*}
	-g^{n}(x)=-g(x) |g(x)|^{n-1} \ge 0
	\end{align*}
	and obviously (as $ n+1 $ is then even)
	\begin{align*}
	g^{n+1}(x)\ge 0.
	\end{align*}
	We obtain using the arguments collected  before
	\begin{align*}
	\binom{{p-1 }}{ n}&\bigl(g^{n}(-x)-g^{n}(x)\bigr)+\binom{{p-1 }}{ n+1}\bigl(g^{n+1}(-x)-g^{n+1}(x)\bigr)\\
	&=\abs{\binom{{p-1 }}{ n}}\bigl(g^{n}(-x)-g^{n}(x)\bigr)-\abs{\binom{{p-1 }}{ n+1}}\bigl(g^{n+1}(-x)-g^{n+1}(x)\bigr)\\
	&\ge\abs{\binom{{p-1 }}{ n}}g^{n}(-x)-\abs{\binom{{p-1 }}{ n+1}}g^{n+1}(-x)\\
	&\ge \abs{\binom{{p-1 }}{ n}}\left(g^{n}(-x)-g^{n+1}(-x)\right)\\
	&\ge 0,
	\end{align*}
	where the last inequality follows from 
	$ 0\le g(-x)< 1 $ for $0\le x\le1/2 $, see Lemma~\ref{l:prop_g}.
\end{proof}

With Lemma~\ref{l:diff} we can treat the case of $ p\ge 3 $ lying between an odd and an even number.
This is done in the next proposition.
\begin{proposition}\label{p:Case1}
	Let $ p \ge 3 $ be such that there is 
	$k\in\NN$  with $ k\geq 2$ and $ 2k-1\leq p\leq 2k $. Then, for all $0< x\leq 1/2$  we have $ F_{p}(x)\ge 0 $ and
	\begin{align*}
	E_{p}(x)+F_{p}(x)>0.
	\end{align*}
	In particular, $ w_{p}(n)>w_{p}^{H}(n) $ for $ n\ge2 $. 
\end{proposition}
\begin{proof}
	We can write $ F_{p}(x)=\sum_{n=2}^{\infty}\binom{p-1}{n}\bigl(g^{n}(-x)-g^{n}(x)\bigr)$ as
	\begin{align*}
	F_{p}(x)=&\binom{p-1}{2}\bigl(g^{2}(-x)-g^{2}(x)\bigr)\\
	&+\sum_{n\in 2\NN+1}^{\infty} \left(\binom{p-1}{n}\bigl(g^{n}(-x)-g^{n}(x)\bigr) + \binom{p-1}{n+1} \bigl(g^{n+1}(-x)-g^{n+1}(x)\bigr)\right)
	\end{align*}
	By Lemma~\ref{l:diff} the terms in the sum on the right hand side are all positive. Furthermore, $ \binom{p-1}{2}\ge0 $ for $ p\ge 3 $ and $ g(-x)\ge |g(x) |$ by Lemma~\ref{l:prop_g}. Thus, also the first term on the right hand side is positive as well and $ F_{p}\ge0 $ follows. From the discussion in the beginning in the section we take $ E_{p}(x)>0 $ for $ 0<x \leq 1/2 $. The ''in particular`` follows from the discussion above  Lemma~\ref{l:prop_g}.
\end{proof}

Note that we cannot treat the case $ 1\leq p\leq 2 $ since the sum in $ F_{p} $ starts at the index $ n=2 $. Hence, there is still a negative term $ \binom{p-1}{2}\bigl(g^{2}(-x)-g^{2}(x)\bigr).$  We deal with this case,    $ 1\leq p\leq 2 $, next.

To this end, we denote the Taylor coefficients of $x\mapsto g(-x) $ by $ a_{k} $, 
i.e.,
\begin{align*}
g(-x)&=q\sum_{k=1}^{\infty}\binom{{1/q }}{ k+1}(-x)^{k} = \sum_{k=1}^{\infty}a_{k}x^k,\\
g(x)&=q\sum_{k=1}^{\infty}\binom{{1/q }}{ k+1}x^{k} = \sum_{k=1}^{\infty}a_{k}(-1)^kx^k.
\end{align*}
The function $ E_{p}(x)=-2p\sum_{k\in 2\mathbb{N}+1}\binom{{1/q }}{ k+1}x^{k} $ is odd and, therefore, we have
\begin{align*}
E_{p}(x)=2(p-1)\sum_{n=1}^{\infty}a_{2n+1}x^{2n+1}.
\end{align*}

\begin{lemma}\label{l:gpm} Let $ p\ge 1 $ and $ 0\leq x\leq 1/2 $. Then,
	\begin{align*}
	g(-x)+g(x)\leq \frac{p+1}{9p^2}.
	\end{align*}
\end{lemma}
\begin{proof}
	We calculate using $ a_{2}\ge a_{n} $ for $ n\ge 2$ and $ 1-1/q=1/p $
	\begin{align*}
	g(-x)+g(x)&=2\sum_{k=1}^\infty a_{2k}x^{2k}\leq 2a_{2}\sum_{k=1}^\infty 2^{-2k}=\frac{2}{3}q\left|\binom{1/q}{3}\right|
	=\frac{p+1}{9p^2}.\qedhere
	\end{align*}
\end{proof}

With the help of this lemma and Lemma~\ref{l:diff} we can treat the case $ p\in (1,2] $.

\begin{proposition}\label{p:Case2}
	Let $ p \in (1,2] $. Then, for all $0< x\leq 1/2$,  we have 
	\begin{align*}
	E_{p}(x)+F_{p}(x)>0.
	\end{align*}
	In particular, $ w_{p}(n)>w_{p}^{H}(n) $ for $ n\ge2 $. 
\end{proposition}
\begin{proof}We show $ E_{p}+F_{p}>0 $ and deduce the ''in particular'' from the discussion above Lemma~\ref{l:prop_g}.
	By Lemma~\ref{l:diff} we have for all $0< x\leq 1/2$
	\begin{align*}
	E_{p}(x)&+F_{p}(x)\\
	=&E_{p}(x)+\binom{p-1}{2}\bigl(g^{2}(-x)-g^{2}(x)\bigr)\\
	&+\sum_{n\in 2\NN+1}^{\infty} \left(\binom{p-1}{n}\bigl(g^{n}(-x)-g^{n}(x)\bigr) + \binom{p-1}{n+1} \bigl(g^{n+1}(-x)-g^{n+1}(x)\bigr)\right)
	\\\ge& E_{p}(x)+\binom{p-1}{2}\bigl(g^{2}(-x)-g^{2}(x) \bigr)\\
	=& E_{p}(x)+\frac{(p-1)(p-2)}{2}\bigl(g(-x)-g(x) \bigr)\bigl(g(-x)+g(x) \bigr).
	\end{align*}
	We proceed  using the definition of $ E_{p} $, i.e.,  $(p-1)\bigl(g(-x)-g(x)\bigr)=E_{p}(x)+ \frac{p-1}{p}x $, to obtain
	\begin{align*}
	\ldots
	&=E_{p}(x)+\frac{p-2}{2}\left(E_{p}(x)+\frac{p-1}{p}x \right)\bigl(g(-x)+g(x) \bigr)\\
	&=\left(1+\frac{p-2}{2}\bigl(g(-x)+g(x) \bigr)\right)E_{p}(x)+(p-1)\frac{(p-2)}{2p}x\bigl(g(-x)+g(x) \bigr)\\
	&\geq {\left(1+\frac{(p-2)(p+1)}{18p^2}\right)}E_{p}(x)+(p-1){\frac{p-2}{2p} }x\bigl(g(-x)+g(x) \bigr),
	\end{align*}
	where we used Lemma~\ref{l:gpm} and $ E_{p}\ge0 $ to gain the inequality. Hence, we get using the representation $ E_{p}(x)=2(p-1)\sum_{k=1}^{\infty}a_{2k+1}x^{2k+1} $ and $ \bigl(g(-x)+g(x)\bigr)=2\sum_{k=1}^{\infty}a_{2k}x^{2k} $
	\begin{align*}
	\ldots=&2(p-1)\sum_{k=1}^\infty \left(\left(1+\frac{(p-2)(p+1)}{18p^2}\right)a_{2k+1}+\frac{p-2}{2p}a_{2k}\right)x^{2k+1}.
	\end{align*}
	Using the fact $ a_{2k+1}=\frac{q(2k+1)-1}{q(2k+2)}a_{2k}\ge\frac{3q-1}{4q}a_{2k}=\frac{2p+1}{4p}a_{2k}  $  we deduce
	\begin{align*}
	\ldots\geq&2(p-1)\sum_{k=1}^\infty \left(\left(1+\frac{(p-2)(p+1)}{18p^2}\right)\frac{2p+1}{4p}+\frac{p-2}{2p}\right)a_{2k}x^{2k+1}\\
	=&2(p-1)  \frac{74 p^3 - 55 p^2 - 5 p - 2}{72 p^3}\sum_{k=1}^\infty a_{2k}x^{2k+1}\ge 0,
	\end{align*}
	where positivity follows for $ p>1 $ as the leading coefficient of $p\mapsto 74p^3-55p^2-5p-2 $ is larger than the sum of the absolute values of the remaining coefficients.
\end{proof}

Hence, it remains to consider the case of $ p $  between an even and an odd integer for which we need the following three lemmas.
\begin{lemma}\label{l:q}
	Let $ p, q \geq 1 $ such that $1/p+1/q=1$ and $ k\ge 2 $. Then,
	\begin{align*}
	a_{k}=q\left|	\binom{1/q}{k+1}\right|\ge\frac{1}{pk(k+1)}=\frac{1}{q(p-1)k(k+1)}.
	\end{align*}
\end{lemma}
\begin{proof}
	We calculate using $1/p+1/q=1$
	\begin{align*}
	q\left|	\binom{1/q}{k+1}\right|&=\frac{(1-1/q)(2-1/q)(3-1/q)\cdots (k-1/q)}{(k+1)!} \\
	&=\frac{1}{pk(k+1)} \frac{(1+1/p)(2+1/p)\cdots ((k-1)+1/p)}{(k-1)!} \\
	&= \frac{1}{pk(k+1)} \left(1+\frac{1}{p}\right)\left(1+\frac{1}{2p}\right)\cdots \left(1+\frac{1}{(k-1)p}\right) \geq \frac{1}{pk(k+1)}.\qedhere
	\end{align*}
\end{proof}

\begin{lemma}\label{l:p}
	Let $ p, q\in(1,\infty) $ such that $1/p+1/q=1$ and $k\in\NN, k> p $. Then,
	\begin{align*}
	\left|\binom{p-1}{k}\right|\leq  \frac{1}{4(p-1)}=\frac{(q-1)}{4}.
	\end{align*}
\end{lemma}
\begin{proof}
	Let $ n\in \NN  $ be such that $ n-1\leq p\leq n $. Moreover, let $\gamma= p-(n-1)$, i.e., $1-\gamma=n-p$, so, $ \gamma\in [0,1] $. Since $ k>p $ and $n,k\in\NN  $, we have that $ k\ge n $ and therefore,
	\begin{align*}
	\left|\binom{p-1}{k}\right|&= \left|\frac{(p-1)(p-2)\cdots(p-(n-1)) (p-n)\cdots (p-k)}{k!}\right|\\
	&=\left|{\left(\frac{p-1}{n-1}\right)\left(\frac{p-2}{n-2}\right)\cdots\frac{\bigl({p}-(n-1)\bigr) }{1}\left(\frac{p-n}{n}\right)\cdots \left(\frac{p-k}{k}\right)}\right|\\
	&\leq \left| \frac{({p}-(n-1))(p-n)}{n}\right|=\frac{\gamma(1-\gamma)}{n} \leq \frac{1}{4(p-1)}=\frac{(q-1)}{4}.\qedhere
	\end{align*}
\end{proof}

\begin{lemma} \label{l:g}
	For $0< x\leq 1/2 $ and $ q>1  $, we get
	\begin{align*}
	g(-x)\leq \frac{(q-1)(5q-1)}{6q^2}x.
	\end{align*}
\end{lemma}
\begin{proof}
	We calculate using $a_2\geq a_k$ for $k\geq 2$	
	\begin{align*}
	g(-x)
	&= q\left(\left|\binom{{1/q }}{ 2} \right|+ \sum_{k=1}^{\infty}\left|\binom{{1/q }}{ k+2} \right|x^{k}\right)x\\
	&\leq q\left(\left|\binom{{1/q }}{ 2} \right|+ \sum_{k=1}^{\infty}\left|\binom{{1/q }}{ k+2} \right|2^{-k}\right)x\\
	&\leq q\left(\left|\binom{{1/q }}{ 2} \right|+ \left|\binom{{1/q }}{ 3} \right|\right)x\\
	&=\frac{(q-1)(5q-1)}{6q^2}x.\qedhere
	\end{align*}
\end{proof}

With the help of these lemmas we can finally treat the case where $ p $ lies between an even and an odd number.

\begin{proposition}\label{p:Case3}
	Let $ p \in [2,\infty) $ such that there is $ k\in\NN $ with $ 2k\leq p\leq 2k+1 $. Then, for all $0< x\leq 1/2$  we have 
	\begin{align*}
	E_{p}(x)+F_{p}(x)>0.
	\end{align*}
	In particular, $ w_{p}(n)>w_{p}^{H}(n) $ for $ n\ge2 $. 
\end{proposition}
\begin{proof}
	Clearly, we have
	$   \binom{{p-1 }}{ n} \ge0 $
	for $ n\le 2k\le p$ and for $ n\in2\NN $. Since we have $ g(-x)\ge |g(x)| $ by Lemma~\ref{l:prop_g}, we obtain for the first $ n\leq p $ terms  and the  terms for even $ n $ in $ F_{p}(x) $ that \[ \binom{{p-1 }}{ n}  \bigl(g^{n}(-x)-g^{n}(x)\bigr)\ge0 . \]
	Note that $ E_{p}(x)= 2(p-1)\sum_{n=1}^{\infty}a_{2n+1}x^{2n+1}> 2(p-1)\sum_{n=k}^{\infty}a_{2n+1}x^{2n+1}$ since the coefficients $a_k$ are positive. With the observation made at the beginning of the proof, this leads to 
	\begin{align*}
	E_{p}(x)+F_{p}(x)> \sum_{n\in 2\NN+1, n\ge 2k+1}\left(2(p-1)a_{n}x^{n}+\binom{p-1}{n}\bigl(g^{n}(-x)-g^{n}(x)\bigr)\right).
	\end{align*}
	For $ n\ge 2k+1 $ with $ n\in 2\NN +1 $, we use $ g(-x) \ge |g(x)| $, Lemma~\ref{l:prop_g}, as well as $\binom{p-1}{n} \leq 0$ in order to estimate
	\begin{align*}
	2(p-1)a_{n} x^{n}&+\binom{p-1}{n} \bigl(g^{n}(-x)-g^{n}(x)\bigr)\ge 2(p-1)a_{n} x^{n}+2\binom{p-1}{n} g^{n}(-x).
	\end{align*}
	We use the estimate on $ a_{n} $, Lemma~\ref{l:q}, on $ \binom{p-1}{n} $, Lemma~\ref{l:p}, and the estimate on $ g(-x) $, Lemma~\ref{l:g} in order to get
	\begin{align*}
	\ldots&\ge 
	2\left(\frac{1}{qn(n+1)}-\frac{(q-1)}{4}\left(\frac{(q-1)(5q-1)}{6q^2}\right)^{n}\right)x^{n}.
	\end{align*}
	Since $ p\ge 2\ge q $, the minimum  is clearly assumed at $ q=2 $. Thus
	\begin{align*}	
	\ldots&\ge\left(\frac{1}{n(n+1)}-\frac{1}{2}\left(\frac{3}{8}\right)^{n}\right)x^{n}\ge\left(\frac{1}{n(n+1)}-\frac{1}{2^{n+1}}\right)x^{n}>0,
	\end{align*}
	where the positivity follows by a simple induction argument. This concludes the proof by noticing that the ''in particular`` part follows from the discussion above Lemma~\ref{l:prop_g}.
\end{proof}

In summary, the above considerations yield 
\begin{align*}
E_{p}(x)+F_{p}(x)>0
\end{align*}
for $ p\in (1,\infty) $ and $0< x\leq 1/2 $. By the discussion at the beginning of the section this yields $ w_{p}(n)>w_{p}^{H}(n) $ 
for $ n\ge 2 $.

We finish the section by treating the case $ n=1 $ which corresponds to $ x=1 $. With this we finally conclude that 
$ w_{p}(n)>w_{p}^{H}(n) $
for all $ n\ge1 $ in the next section.

\begin{proposition}\label{p:1}Let $ p\in(1,\infty) $. Then, $w_p(1)> w^H_p(1)$.
\end{proposition}
\begin{proof}
	Recall that $ w_p(1) =1-(2^{1-1/p}-1)^{p-1}$ and $ w_{p}^{H}=(1-1/p)^{p} $. By the mean value theorem applied to the function $ [1,2]\to[1,2^{1-1/p}] $, $ t\mapsto t^{1-1/p} $ we find
	\begin{align*}
	2^{1-1/p}-1<1-\frac{1}{p}.
	\end{align*}
	Therefore,
	\begin{align*}
	w_{p}(1)-w_{p}^{H}(1)>1-\left(1-\frac{1}{p}\right)^{p-1}-\left(1-\frac{1}{p}\right)^{p}.
	\end{align*}
	Now the function $\psi\colon (1,\infty) \to (0,\infty)$, $ p\mapsto \left(1-{1}/{p}\right)^{p-1}+\left(1-{1}/{p}\right)^{p} $ is strictly monotonically decreasing because
	\begin{align*}
	\psi'(p)= \frac{1}{p-1}\left(\frac{p - 1}{p}\right)^p \left((2 p - 1) \log\left(\frac{p - 1}{p}\right) + 2\right)<0,
	\end{align*}
	since $\theta\colon p\mapsto (2p-1)\log (p-1)/p  $ is strictly monotonically increasing and we have $ \lim_{p\to\infty}\theta(p) = -2 $. Hence, we conclude
	\begin{align*}
	w_{p}(1)-w_{p}^{H}(1)> 1-\psi(p) > 1-\lim_{t\to1}\psi(t)=0.
	\end{align*}
	This finishes the proof.
\end{proof}

\section{Proof of Theorem~\ref{thm}}\label{sec:proof}
\begin{proof}[Proof of Theorem~\ref{thm}] Combining Proposition~\ref{prop} and Proposition~\ref{p:w_p} from Section~\ref{sec:Hardy} yields that $ w_{p} $ satisfies the Hardy inequality.  In Section~\ref{sec:w_p>w_H}, one obtains from 	Proposition~\ref{p:Case1}, Proposition~\ref{p:Case2}, Proposition~\ref{p:Case3} and Proposition~\ref{p:1} that $ w_{p}>w_{p}^{H} $ on $\NN$. 
	
To see the statement about the coefficients in the series expansion of $ w_{p} $ for integer $p \geq 2$, recall the function 
\[  w(x)=\bigl(1-(1-x)^{1/q}\bigr)^{p-1}-\bigl((1+x)^{1/q}-1\bigr)^{p-1}  \]
with $ 1/p+1/q=1 $ on $ [0,1] $ from the previous section. It is easy to check that the function $w_{+}\colon x\mapsto {\bigl(1-(1-x)^{1/q}\bigr)^{p-1}}$ is  absolutely monotonic  on $ [0,1) $ for integer $ p\ge 2 $ with strictly positive derivatives. On the other hand, expanding the function $w_{-}\colon  x\mapsto {\bigl((1+x)^{q}-1\bigr)^{p-1}} $ at $ x=0 $, we observe that it has the same Taylor coefficients in absolute value as $ w_{+} $. However, the signs alternate such that for the difference $ w_{+}-w_{-} $ of these two functions the even/odd coefficients cancel for odd/even $ p $. Furthermore, in the Taylor expansion of $ w $ at $ x=0 $ the first non-zero coefficient is the one  for $ x^{p} $ (confer Section~\ref{sec:w_p>w_H}).
 This finishes the proof.
\end{proof} 


\noindent\textbf{Acknowledgements.} The authors gratefully acknowledge inspiring discussions with Y.~Pinchover. Moreover, the authors are grateful for the financial support of the DFG. Furthermore, F.~F. thanks the Heinrich-Böll-Stiftung for the support.

\bibliographystyle{alpha}
\bibliography{literature}

\end{document}